\documentclass{amsart}

\usepackage{amsthm}
\usepackage{amsfonts}
\usepackage{amssymb}
\usepackage{mathrsfs}

\newtheorem{thm}{Theorem}
\newtheorem{prop}{Proposition}
\newtheorem{lem}{Lemma}
\newtheorem{cor}{Corollary}
\DeclareMathOperator{\Span}{Span}
\DeclareMathOperator{\ch}{char}

\title{Cubic surfaces with one rational line over finite fields}
\author{Jenny Cooley}
\address{Mathematics Institute\\University of Warwick\\ Coventry\\ CV4 7AL\\UK}
\email{j.a.cooley@warwick.ac.uk}
\date{\today}
\thanks{The author is supported by an Engineering and Physical Sciences Research Council (EPSRC) studentship}


\begin{document}

\begin{abstract}
Let $\mathbb{F}_q$ be a finite field with $q=8$ or $q\geq 16$. Let $S$ be a smooth cubic surface defined over $\mathbb{F}_q$ containing at least one rational line. We use a pigeonhole principle to prove that all the rational points on $S$ are generated via tangent and secant operations from a single point.
\end{abstract}

\maketitle

\section{Introduction} 
Let $S$ be a smooth cubic surface defined over a field $K$. One can define a secant and 
tangent process of generating new rational points from old. This process is somewhat 
analogous to the group operation on the rational points of an elliptic curve. The secant operation is defined as follows. Let 
$P$, $Q\in S(K)$ and $P\neq Q$. Let $\ell$ be the line
joining $P$ and $Q$. If $\ell \not \subset S$ then $\ell$ intersects $S$ in exactly
three points counting multiplicities, i.e. $\ell\cdot S=P+Q+R$ where $R \in S(K)$.
If $R$ is distinct from $P$ and $Q$ then we say we have generated $R$ from $P$ and $Q$.
For the tangent operation let $\Pi_P$ be the tangent plane to $S$ at $P$. Let $\ell$ be 
any $K$-line lying in $\Pi_P$. If $\ell\not\subset S$ then $\ell\cdot S=2P+R$. If $R \neq P$
then we have generated $R$ from $P$.

Although the set of rational points on a cubic surface is not a group nor is the secant and 
tangent process a binary operation, one can still pose questions about the size of a minimal
generating set of points. Problems of this type were first studied by Segre in~\cite{Segre},
and by Manin in~\cite{Ma1}. In order to study this problem we define the $\Span$ of a set of 
points in $S(K)$ as follows.

Let $\Sigma$ be a set of $K$-points on $S$. We define $\Span(\Sigma)$ as the minimal set
of points in $S(K)$ that is closed under the secant and tangent operations. So we have $\Sigma\subseteq\Span(\Sigma)\subseteq S(K)$. Note that we will write $\Span(P_1,\dots,P_n)$ for $\Span(\{P_1,\dots,P_n\})$
for ease of notation.

In~\cite{Cooley} the following theorem was proven.
\begin{thm}
\label{skew}
Let $K$ be a field with $\# K\geq 4$. Let $S$ be a smooth cubic surface defined over $K$ containing a skew pair of $K$-lines, $\ell$ and $\ell^\prime$. Let $P\in S(K)$ be a non-Eckardt point belonging to either $\ell$ or $\ell^\prime$. Then
\[
 Span(P)=S(K).
\]
\end{thm}

In this paper I will prove the following Theorem.
\begin{thm}
\label{oneline}
Let $K=\mathbb{F}_q$ be a finite field with $q=8$ or $q\geq 16$. Let $S$ be a smooth cubic surface defined over $K$ containing at least one $K$-rational line. Then there exists a point $P\in S(K)$ such that $\Span(P)=S(K)$.
\end{thm}
This theorem extends Theorem~\ref{skew} in the case where $K=\mathbb{F}_q$ and $q=8$ or $q\geq16$. 
Over $\mathbb{Q}$ there is a family of cubic surfaces, each containing a $K$-line, for which there exists no point $P\in S(K)$ such that $\Span(P)=S(K)$, and in fact the number of points required to generate $S(K)$ for this family is unbounded. Thus the behaviour of the secant and tangent process over finite fields is notably different to that over $\mathbb{Q}$~\cite{Siksek}.

I would like to thank Daniel Loughran, Samir Siksek, Efthymios Sofos and Damiano Testa for many useful and enlightening discussions during the preparation of this paper.


\section{Useful results} 
\label{intro}
Throughout this paper we will refer to the tangent plane to $S$ at a point $P$ as $\Pi_P$. 
We also define $\Gamma_P=\Pi_P\cdot S$. This is a cubic curve that is singular at $P$ and may or may not be irreducible.
The \emph{Gauss map} on $S(\overline{K})$ is the map that takes a point $P$ to its tangent plane $\Pi_P$. We define $\gamma_\ell$ to be the restriction of the Gauss map to a $K$-line $\ell\subset S$. Before discussing some further properties of the 
Gauss map, we state the following lemma; a proof of which can be found in~\cite{Siksek}.
\begin{lem}
\label{ellingamP}
Let $P\in S(\overline{K})$. The curve $\Gamma_P$ contains every $\overline{K}$-line on $S$ that passes
through $P$.
\end{lem}
This implies that if $P\in\ell\subset S$ then $\ell\subset\Pi_P$. As discussed in~\cite{Siksek}, 
$\gamma_\ell$ has degree~$2$ and we define a \emph{parabolic point} in $\ell(\overline{K})$ to be a point
of ramification of $\gamma_\ell$. A special case of a parabolic point is an \emph{Eckardt point}, which is a point where
three lines in $S$ meet. By Lemma~\ref{ellingamP} all three lines must be coplanar, and since $\Gamma_P$ has degree $3$, we cannot have more than $3$ lines in $S$ meeting in a single point. The following lemma, also proved in~\cite{Siksek}, gives useful 
information on the geometry of $\Gamma_P$ for $P\in\ell(\overline{K})$.
\begin{lem}
\label{gammapts}
Let $\ell$ be a $K$-line contained in $S$.
\renewcommand{\theenumi}{\roman{enumi}}
\begin{enumerate}
 \item  If $\ch(K) \ne 2$ then $\gamma_\ell$ is separable. Precisely two points $P\in\ell(\overline{K})$ are
         parabolic, and so there are at most two Eckardt points on $\ell$.
\item If $\ch(K) = 2$ and $\gamma_\ell$ is separable then there is precisely one point $P\in\ell(\overline{K})$
         which is parabolic and so at most one Eckardt point on $\ell$.
\item If $\ch(K) = 2$ and $\gamma_\ell$ is inseparable then every point $P\in\ell(\overline{K})$ is parabolic
         and the line $\ell$ contains exactly $5$ Eckardt points.
\end{enumerate}
\end{lem}


The degree of $\gamma_\ell$ is $2$ so non-parabolic points in $\ell(K)$ come in pairs. Therefore for $\mathbb{F}_q$ with $q$ odd there are either zero or two parabolic points in $\ell(K)$. For $q$ even there is exactly one parabolic point in $\ell(K)$. This follows from $\ell$ being a copy of $\mathbb{P}_K^1$ and hence $|\ell(K)|=q+1$.

The configuration of the $27$ lines on a smooth cubic surface~\cite{Reid} implies that when a smooth cubic surface contains four $K$-lines, it must contain a skew pair. Therefore we need only prove Theorem~\ref{oneline} for cubic surfaces containing no more than three $K$-lines, all of which must be coplanar.

The following result was proved in~\cite{Cooley}.
\begin{lem}
\label{gengamma}
Let $K$ be a field with at least $4$ elements and $S$ a smooth cubic surface defined over $K$. Let $\ell$ be a $K$-line contained in $S$. Let $P \in \ell(K)$ and suppose $P$ is not an Eckardt point. Then
\[
\ell(K) \subseteq \Gamma_P(K) \subseteq \Span(P).
\]
\end{lem}

\section{Proof of Theorem~\ref{oneline}}
First note that any $K$-line $\ell$ in $\mathbb{P}_K^3$ is a copy of $\mathbb{P}_K^1$
and hence when $K=\mathbb{F}_q$ we have $|\ell(K)|=q+1$. By the configuration of the lines on a smooth cubic surface there are exactly ten
$\overline{K}$-lines in $S$ intersecting any given $\overline{K}$-line in $S$~\cite{Reid}. Thus there are at most $5$ Eckardt 
points in $\ell(K)$. This along with Lemma~\ref{gammapts} implies that when $q=3$ or $q\geq 5$ we are guaranteed the existence of a non-Eckardt point in $S(K)$.

Let $\ell\subset S$ be a $K$-line. There is a pencil of planes passing through $\ell$, hence there is a bijection between rational points in $\mathbb{P}_K^1$ and $K$-planes through $\ell$. Therefore we know that there are exactly $q+1$ distinct $K$-planes through $\ell$. Let $\Pi$ be any such plane and $\Gamma=\Pi\cdot S$. Then $\Gamma=\ell\cup C$ where $C$ is some conic defined over $K$ that may be absolutely irreducible or may decompose into two lines. The possible cases are enumerated below.
\begin{enumerate}
\item \label{ratconic} $C$ is absolutely irreducible and meets $\ell$ in two distinct $K$-points.
\item \label{ratlines} $C$ decomposes into two $K$-lines, $m$ and $m^\prime$, and $m\cdot\ell \ne m^\prime\cdot\ell$.
\item \label{irconic} $C$ is absolutely irreducible and meets $\ell$ in two distinct points that are not $K$-points, but are defined over $\mathbb{F}_{q^2}$.
\item \label{irlines}$C$ decomposes into two Galois conjugate lines $m$, $m^\prime$, that are defined over $\mathbb{F}_{q^2}$ but not $K$. We have $m\cdot m^\prime$ is a $K$-point, $m\cdot\ell\ne m^\prime\cdot\ell$, and $m\cdot\ell$, $m^\prime\cdot\ell$ are defined over $\mathbb{F}_{q^2}$ but are not $K$-points.
\item \label{parabconic} $C$ is absolutely irreducible and $\ell$ is tangent to $C$, so meeting $C$ in exactly one $K$-point.
\item \label{ratEck} $C$ decomposes into two $K$-lines, $m$ and $m^\prime$, and $m\cdot\ell=m^\prime\cdot\ell$, and hence is an Eckardt point in $\ell(K)$.
\item \label{irEck} $C$ decomposes into two Galois conjugate lines $m$, $m^\prime$, that are defined over $\mathbb{F}_{q^2}$ but not $K$. We have $m\cdot\ell=m^\prime\cdot\ell$, and hence is an Eckardt point in $\ell(K)$.
\end{enumerate}
Note that in case~(\ref{irEck}) we have $\Gamma(K)=\ell(K)$ since there are no $K$-points on $m$ and $m^\prime$ other than the intersection point $m\cdot\ell=m^\prime\cdot\ell=m\cdot m^\prime$.

In cases (\ref{parabconic}), (\ref{ratEck}) and (\ref{irEck}) the conic $C$ intersects the line $\ell$ in a single point. Such a point $P$ is a parabolic point and $\Pi$ is the tangent plane to $S$ at $P$ denoted $\Pi_P$. The number of parabolic points on $\ell$ is determined by whether the Gauss map on $\ell$ is separable or inseparable as described in section~\ref{intro}. 
We deal with the case where $\ch(K)=2$ and $\gamma_\ell$ is inseparable separately. In this case
we can prove Theorem~\ref{oneline} directly without invoking a pigeonhole principle; indeed this result also holds for infinite fields of characteristic $2$.
\begin{lem}
\label{gaminsep}
Let $K$ be a field of characteristic $2$ containing at least $8$ elements. Let $S$ be a smooth cubic surface defined over $K$ containing at least one $K$-line, $\ell$, upon which the Gauss map, $\gamma_\ell$, is inseparable. Let $P\in\ell(K)$ be a non-Eckardt point. Then
\[
\Span(P)=S(K).
\]
\end{lem}
\begin{proof}
Since $\gamma_\ell$ is inseparable every $K$-point on $\ell$ is parabolic. This means that every point in $S(K)$ belongs to $\Gamma_P(K)$ for some point in $\ell(K)$. By the configuration of the $27$ lines on a cubic surface~\cite{Reid}, there are at most five Eckardt points in $\ell(K)$. If more than one of these Eckardt points is of type~(\ref{ratEck}) then we have more than three $K$-lines in $S$ and hence a skew pair. Thus we apply Theorem~\ref{skew} to obtain the result. Therefore we may assume that there is at most one Eckardt point of type~(\ref{ratEck}) in $\ell(K)$. Let $P\in\ell(K)$ be a non-Eckardt point. By Lemma~\ref{gengamma} we have $\ell(K)\subset\Gamma_P(K)\subset\Span(P)$. Applying Lemma~\ref{gengamma} again gives us that $\Gamma_Q(K)\subset\Span(P)$ for all non-Eckardt points $Q\in\ell(K)$. Hence if there are no Eckardt points of type~(\ref{ratEck}) in $\ell(K)$ then we have generated all of the points in $S(K)$. Let us then suppose that there is an Eckardt point of type~(\ref{ratEck}) in $\ell(K)$, which we will denote $E$. Again by Lemma~\ref{gengamma} we know that
\[
\bigcup_{R\in\ell(K)\setminus\{E\}}\Gamma_R(K)\subset\Span(P).
\]
Let $G:=S(K)\setminus\Gamma_E(K)$. Observe that $G\subset\bigcup_{R\in\ell(K)\setminus\{E\}}\Gamma_R(K)\subset\Span(P)$. Let $Q\in\Gamma_E(K)\setminus\ell$. Note that all points $Q^\prime\in G$ lie on some conic $C$ of type~(\ref{parabconic}), that is in $\Gamma_{P^\prime}(K)\setminus\ell$ for some $P^\prime\in\ell(K)$. The intersection $\Pi_Q\cdot\Pi_{P^\prime}$ is a $K$-line in $\mathbb{P}_K^3$, not contained in $\Pi_E$, therefore $\Pi_Q\cdot\Pi_{P^\prime}$ is not contained in $S$, and hence intersects $C$ exactly twice counting multiplicities. Thus the set $G\setminus\Gamma_Q(K)$ is non-empty. Let $Q^\prime\in G\setminus\Gamma_Q(K)$. Let $m$ be the line joining $Q$ and $Q^\prime$. We know that $m\not\subset S$ since $m\not\subset\Gamma_E$. Therefore $m\cdot S=Q+Q^\prime+R$, where $R\ne Q$ and hence $R\not\in\Pi_E$. From whence $R\in G$, which gives us 
$\Gamma_E(K)\setminus\ell\subset\Span(P)$
and
$\ell(K)\subset\Span(P)$,
from which the result follows. 
\end{proof}

The following proposition explains the pigeonhole principle required to prove Theorem~\ref{oneline} for the remaining cases.
\begin{prop}
\label{pigeon}
Let $S$ be a smooth cubic surface over $K=\mathbb{F}_q$. Let $T\subseteq S(K)$ be such that $\ell(K)\subset T$ for all $K$-lines $\ell$ contained in $S$, and  $|T|>\frac{1}{2}|S(K)|+\frac{q+1}{2}$. Then 
\[
\Span(T)=S(K).
\]
\end{prop}
\begin{proof}
Suppose we wish to generate a point $Q\not\in T$ and hence not lying on any $K$-line in $S$. In order to do so we require points $R$, $R^\prime\in T$ and a $K$-line $m$ such that $m\cdot S=R+R^\prime+Q$. Note that we may have $R=R^\prime$, in which case $Q$ is generated via a tangent operation on $R$. We also know that $Q$ and $R$ are distinct since $R\in T$ and $Q\in T^\prime=S(K)\setminus T$, and hence they uniquely determine $m$. Further note that a point $R\in\Gamma_Q(K)$ can never belong to a pair of points $R$, $R^\prime$ that generate $Q$ since then we would have $m\cdot S=2Q+R+R^\prime$, which would contradict the fact that $m$ intersects $S$ exactly three times counting multiplicities. We therefore define the following sets of points in $S(K)$.
\[
\begin{array}{l}
G = T\setminus\Gamma_Q(K) \\
B = T^\prime\setminus\{Q\}.
\end{array}
\]
The idea of the proof is as follows: we will define a map $\phi :G \rightarrow B$ and show that if $Q$ cannot be generated from the points in $G$ (and hence $T$) then $\phi$ is injective. We will then show that if $|T|>\frac{1}{2}|S(K)|+\frac{q+1}{2}$ then $|G| > |B|$ contradicting the injectivity of $\phi$. It will then follow that $S(K)=\Span(G)$ and since $G\subset T$ that $S(K)=\Span(T)$.

Suppose $Q\not\in\Span(G)$. Let $R\in G$ and $m$ be the unique $K$-line joining $Q$ and $R$. Since $G = T\setminus\Gamma_Q(K)$ we know that $m\not\subset S$. Therefore $m\cdot S=Q+R+Q^\prime$ where $Q^\prime\in S(K)$, $Q^\prime\neq Q$ and is uniquely determined by $m$. Since $Q\not\in\Span(G)$ we must have $Q^\prime\in B$. This defines an injective map $\phi : G \rightarrow B$ with $\phi(R)=Q^\prime$.

It now remains to show that $|T|>\frac{1}{2}|S(K)|+\frac{q+1}{2}\implies |G| > |B|$. Since $Q$ is a $K$-point and does not lie on any $K$-lines in $S$ we know that either $\Gamma_Q$ is the union of three lines defined over a cubic extension of $K$ that are Galois conjugates and that $Q$ is an Eckardt point, or that $\Gamma_Q$ is an irreducible cubic curve with a singular point at $Q$. The following table gives the possibilities for the size of $\Gamma_Q(K)$~\cite{HirshGorTor}.
\begin{center}
\begin{tabular}{| c | c |}
  \hline                        
  Form of $\Gamma_Q$ & $\#\Gamma_Q(K)$  \\ \hline
 3 $\mathbb{F}_{q^3}$-lines and $Q$ Eckardt & $1$  \\
 cusp at $Q$ & $q+1$  \\
 split node at $Q$ & $q$ \\
 non-split node at $Q$ & $q+2$ \\
  \hline  
\end{tabular}
\end{center}
Recall
\[
 |T|>\frac{1}{2}|S(K)|+\frac{q+1}{2},
\]
and
\[
 |\Gamma_Q(K)|\leq q+2.
\]
From this we obtain
\[
\begin{array}{rcl}
 |G| & = & |T\setminus\Gamma_Q(K)| \\
 & \geq & |T|-(q+2) \\
 & > & \left(\frac{1}{2}|S(K)|+\frac{q+1}{2}\right)-(q+2) \\
 & = & |S(K)|-\left(|S(K)+\frac{q+1}{2}\right)-1 \\
 & > & |S(K)\setminus T|-1 \\ 
 & = & |T^\prime\setminus\{Q\}| \\
 & = & |B|,
\end{array}
\]
which completes the proof.
\end{proof}

It now remains to show that we can generate enough points on $S$ to apply the pigeonhole principle. The following lemmas will be useful. Note that from now on we may assume that the Gauss map on any $K$-line in $S$ is separable.

\begin{lem}
\label{char2genEck}
Let $K$ be a field containing at least $8$ elements. Let $S$ be a smooth cubic surface defined over $K$ containing exactly three $K$-lines, $\ell_1$, $\ell_2$, $\ell_3$, that meet in an Eckardt point $E$ and not all of $\ell_1$ $\ell_2$, $\ell_3$ contain $K$-rational Eckardt points other than $E$. Then
\[
\ell_2(K)\cup\ell_3(K)\subset\Span(\ell_1(K)).
\]
\end{lem}
\begin{proof}
First note that if the characteristic of $K$ is $2$ and $\gamma_{\ell_i}$ is inseparable for any of $i=1,2,3$ then the result follows from Lemma~\ref{gaminsep}. Thus we may assume that $\gamma_{\ell_i}$ is separable for $i=1,2,3$.

Let $\Omega^\prime=\{Q\in S(K) |E\in\Gamma_Q\}$. This is a sextic curve in $\mathbb{P}_K^3$ defined over $K$~\cite{Geiser},~\cite{Baker}. By Lemma~\ref{ellingamP} we have $\ell_1\cup\ell_2\cup\ell_3\subset\Omega^\prime$. We denote by $\Omega$ the cubic curve such that $\Omega^\prime=\Omega\cup\ell_1\cup\ell_2\cup\ell_3$.

Without loss of generality, let $\ell_2$ contain no $K$-rational Eckardt points other than $E$. Let $P\in\ell_1(K)\setminus\{E\}$ be non-Eckardt. Then $\Gamma_{P}=\ell_1\cup C$ where $C$ is an absolutely irreducible conic defined over $K$.

Suppose $\Omega$ is reducible and therefore could contain $C$ as a component. Then $\Omega$ must be the union of a linear and quadratic component, both defined over $K$. This cannot occur because $\Omega\subset S$, $\ell_i\not\subset\Omega$ for $i=1,2,3$, but by the hypotheses of the lemma $\ell_1$, $\ell_2$, $\ell_3$ are the only $K$-lines in $S$. Therefore $\Omega$ is irreducible and does not contain $C$.

By Bezout's Theorem the number of points of intersection between $\Omega\cup\ell_1$ and $C$ is at most $4\cdot2=8$. Therefore $C(K)\setminus(\Omega\cup\ell_1)$ contains at least one point.
Let $Q\in C(K)\setminus(\Omega\cup\ell_1)$. Consider $\Gamma_Q$. This is an absolutely irreducible cubic curve with a singular point at $Q$. Hence $\Gamma_Q(K)\subset\Span(Q)$. The intersection point $R=\Pi_Q\cdot\ell_2\in\Gamma_Q$ is a $K$-point since it is the intersection of a $K$-line and a $K$-plane, and $R\neq E$ since $Q\not\in\Omega^\prime$. Therefore 
\[
R\in \Span(Q)\subset\Span(\ell_1(K)).
\]
By applying Lemma~\ref{gengamma} we see that 
\[
\ell_2\subset\Span(R)\subset\Span(\ell_1(K)).
\]
We perform secant operations on the points of $\ell_1(K)$ and $\ell_2(K)$ to obtain 
\[
\ell_3(K)\subset\Span(\ell_1(K)\cup\ell_2(K)),
\]
from which the result follows.
\end{proof}

\begin{lem}
\label{charoddgenEck}
Let $K$ be a field containing at least $17$ elements. Let $S$ be a smooth cubic surface defined over $K$ containing exactly three $K$-lines, $\ell_1$, $\ell_2$, $\ell_3$, that meet in a $K$-rational Eckardt point $E$. Then
\[
\ell_2(K)\cup\ell_3(K)\subset\Span(\ell_1(K)).
\]
\end{lem}
\begin{proof}
If the characteristic of $K$ is $2$, then the result follows from Lemma~\ref{gaminsep} and Lemma~\ref{char2genEck}. Likewise if any of $\ell_1$, $\ell_2$, $\ell_3$ contains no $K$-rational Eckardt points other than $E$ then the result follows from Lemma~\ref{char2genEck}. Thus we may assume that each of $\ell_i$, $i=1,2,3$, contains precisely one $K$-rational Eckardt point other than $E$, which we will denote $E_i$ respectively. The $E_i$ are all of type(~\ref{irEck}). 

Let $\Omega_E^\prime=\{Q\in S(K) |E\in\Gamma_Q\}$ and $\Omega_{E_2}^\prime=\{Q\in S(K) |E_2\in\Gamma_Q\}$. Let $\Omega_E$ denote the cubic curve such that $\Omega_E^\prime=\Omega_E\cup\ell_1\cup\ell_2\cup\ell_3$ and $\Omega_{E_2}$ denote the cubic curve such that $\Omega_{E_2}^\prime=\Omega_{E_2}\cup\ell_2\cup m\cup m^\prime$ where $m$ and $m^\prime$ are the Galois conjugate lines in $S$ passing through $E_2$. Note that $\Omega_E$ and $\Omega_{E_2}$ are both defined over $K$.

Let $P\in\ell_1(K)$ be a non-Eckardt point. Then $\Gamma_P=C\cup\ell_1$ where $C$ is an absolutely irreducible conic defined over $K$. As in the proof of Lemma~\ref{char2genEck}, $\Omega_E$ cannot contain $C$ since it cannot contain a $K$-rational linear component. Note that $\Omega_{E_2}$ does not contain $\ell_1$ or $\ell_3$ since $E$ is a parabolic point and hence $\Pi_E$ is the tangent plane of no point other than $E$ in $\ell_1\cup\ell_3$. Thus, similarly, $\Omega_{E_2}$ does not contain $C$.

Once again we apply Bezout's Theorem to obtain that the number of points of intersection between $\Omega_{E}\cup\ell_1$ and $C$ is at most $4\cdot2=8$, and the number of points of intersection between $\Omega_{E_2}$ and $C$ is at most $3\cdot2=6$. Therefore $C(K)\setminus(\Omega\cup\ell_1)$ contains at least three points since $K$ contains at least $17$ elements.

Let $Q\in C(K)\setminus(\Omega_E\cup\ell_1\cup\Omega_{E_2})$. Consider $\Gamma_Q$. This is an absolutely irreducible cubic curve with a singular point at $Q$. Hence $\Gamma_Q(K)\subset\Span(Q)$. The intersection point $R=\Pi_Q\cdot\ell_2\in\Gamma_Q$ is a $K$-point since it is the intersection of a $K$-line and a $K$-plane and is not equal to $E$ or $E_2$ since $Q\not\in\Omega_E^\prime\cup\Omega_{E_2}^\prime$. Therefore 
\[
R\in \Span(Q)\subset\Span(\ell_1(K)).
\]
By applying Lemma~\ref{gengamma} we see that 
\[
\ell_2\subset\Span(R)\subset\Span(\ell_1(K)).
\]
We perform secant operations on the points of $\ell_1(K)$ and $\ell_2(K)$ to obtain 
\[
\ell_3(K)\subset\Span(\ell_1(K)\cup\ell_2(K)),
\]
from which the result follows.
\end{proof}

From Lemmas~\ref{gaminsep}, \ref{gengamma}, \ref{char2genEck} and \ref{charoddgenEck} we obtain the following result.
\begin{cor}
\label{genallgam}
Let $K$ be a field containing at least $16$ elements or $K=\mathbb{F}_8$. Let $S$ be a smooth cubic surface defined over $K$
containing at least one $K$-line, $\ell$. Let $P\in\ell(K)$ be a non-Eckardt point. Then
\[
\bigcup_{Q\in\ell(K)}\Gamma_Q(K)\subseteq\Span(P).
\]
\end{cor}


The following lemmas show that there exists a set $T\subset\Span(P)$, for some $P\in S(K)$, such that $|T|>\frac{1}{2}|S(K)|+\frac{q+1}{2}$, thus completing the proof of Theorem~\ref{oneline}.

\begin{lem}
\label{PnotonKline}
Let $S$ be a smooth cubic surface defined over a field $K$ containing at least one $K$-line, $\ell$, and no pair of skew $K$-lines. Let $\Gamma=\Pi\cdot S$ be of type (\ref{ratconic}), (\ref{irconic}) or (\ref{parabconic}) and $\Gamma=C\cup\ell$. Let $P\in C(K)\setminus\ell(K)$. Then $P$ does not lie on a $K$-line in $S$.
\end{lem}
\begin{proof}
Suppose $P\in C(K)$ lies on a $K$-line in $S$, which we will denote by $m$. By the hypotheses of the lemma $m$ cannot be skew to $\ell$ therefore $m$ is coplanar to $\ell$ hence contained in $\Pi$. Thus $m\cup\ell\cup C\subset\Gamma$. But $\Gamma$ is a plane cubic curve so we arrive at a contradiction and $P$ does not lie on a $K$-line in $S$.
\end{proof}



\begin{lem}
\label{Pgengam}
Let $K$ be a field containing at least $13$ elements. Let $S$ be a smooth cubic surface defined over $K$ containing at least one $K$-line, $\ell$, but no pair of skew $K$-lines, and such that the Gauss map on any $K$-line in $S$ is separable. Let $\Pi$ be a plane through $\ell$ such that $\Gamma=\Pi\cdot S$ is of type~(\ref{irconic}) and $\Gamma=C\cup\ell$. Then there exists a point $P\in C(K)$ such that $\Gamma(K)\subset\Span(P)$.
\end{lem}
\begin{proof}

Let $E$ be an Eckardt point in $\ell(K)$ if such a point exists. There are precisely two $\overline{K}$-lines through $E$ in $\Pi$ that are tangent to $C$, i.e. there are precisely two points in $C(\overline{K})$ such that $E$ lies in their tangent planes. As there are at most two Eckardt points in $\ell(K)$ there are at most four points in $C(K)$ that generate an Eckardt point in $\ell(K)$ upon performing a tangent operation.

Let $Q\in C(K)$ be neither an Eckardt point nor a point that has an Eckardt point in $\ell(K)$ in its tangent plane. We know that such a point exists as any Eckardt point in $C(K)$ must be the intersection point of three Galois conjugate lines defined over a cubic extension of $K$. There are a maximum of $26$ such lines, since $\ell$ is a $K$-line, and each can contain no more than one $K$-point. Therefore there are a maximum of $8$ Eckardt points in $C(K)$. Since there are at least $13$ elements in our field we are guarenteed the existence of such a point.

Consider $\Gamma_Q$. This is an absolutely irreducible cubic curve with a singular point at $Q$. The point $R=\Pi_Q\cdot\ell$ is a $K$-point and hence $R\in\Gamma_Q(K)\in\Span(Q)$. Note that $R$ is not Eckardt due to our choice of $Q$. We now apply Lemma~\ref{gengamma} to obtain $\ell(K)\subset\Span(R)\subset\Span(Q)$. Secant operations on $Q$ and the points of $\ell(K)$ give
\[
 C(K)\subset\Span(\ell(K))\subset\Span(Q),
\]
which completes the proof.
\end{proof}

We can now prove the main result. Before doing so it will be useful to note the following values of $\#\Gamma(K)$ for the cubic curves in $S$ of types (1) to (7), which all follow from the fact that lines and absolutely irreducible conics in $S$ are copies of $\mathbb{P}_K^1$ and hence have $q+1$ points.
\begin{center}
\begin{tabular}{| c | c |}
  \hline                        
  Form of $\Gamma$ & $\#\Gamma(K)$  \\ \hline
  (\ref{ratconic}) & $2q$  \\
	(\ref{ratlines}) & $3q$  \\
	(\ref{irconic}) & $2q+2$ \\
	(\ref{irlines}) & $q+2$ \\
  (\ref{parabconic}) & $2q+1$  \\
	(\ref{ratEck}) & $3q+2$ \\
	(\ref{irEck}) & $q+1$ \\
  \hline  
\end{tabular}
\end{center}
\begin{proof}[Proof of Theorem~\ref{oneline}]
When $S$ contains a $K$-line upon which the Gauss map is inseparable the result follows from Lemma~\ref{gaminsep}. Therefore it remains to show that the result holds when  the Gauss map on all $K$-lines in $S(K)$ are separable. We want to show that we can can generate enough points to invoke Proposition~\ref{pigeon} and hence generate all of $S(K)$.

Let $n$ be the number of non-parabolic points in $\ell(K)$. Note that $n$ must be even because $\gamma_\ell$ has degree $2$. In fact we have 
\begin{center}
 \begin{tabular}{ll}
 $n=q$ & when $\ch(K)=2$ \\
 $n=q\pm 1$ & when $\ch(K)\neq 2$
\end{tabular}
\end{center}

There are precisely $q+1$ distinct $K$-planes through $\ell$. For $P\in\ell(K)$ parabolic we have have $\Pi_P$ of type (\ref{parabconic}), (\ref{ratEck}) or (\ref{irEck}). Since $\gamma_\ell$ maps pairs of non-parabolic points in $\ell(K)$ to $K$-planes through $\ell$, of the remaining $K$-planes precisely $\frac{n}{2}$ are of type (\ref{ratconic}) or (\ref{ratlines}), and $\frac{n}{2}$ are of type (\ref{irconic}) or (\ref{irlines}). By the configuration of the $27$ lines in $S$ there are exactly ten $\overline{K}$-lines in $S$ intersecting $\ell$. Hence there are at most five $K$-planes of type~(\ref{irlines}). We are therefore guaranteed the existence of a plane of type~(\ref{irconic}) since $\frac{n}{2} > 5$ for all $q\geq 13$.

Let $\Pi$ be such a plane of type~(\ref{irconic}), let $\Gamma=\Pi\cdot S$ and $\Gamma=C\cup\ell$ where $C$ is an absolutely irreducible conic defined over $K$. We wish to generate a non-Eckardt point in $\ell(K)$ by a tangent operation on a point in $C(K)$. By Lemma~\ref{Pgengam} we know that there exists $P\in C(K)$ such that $\Gamma(K)=C(K)\cup\ell(K)\subset\Span(P)$. By Lemma~\ref{genallgam} we can also generate $\Gamma_Q(K)$ for all $Q\in\ell(K)$. Every other point in $S$ lies in a plane of type~(\ref{irconic}) or (\ref{irlines}). Let $T$ be the set of points that we have already generated and $T^\prime=S(K)\setminus T$.
\[
\begin{array}{rlr}
\# T\geq & 2q+2 & (\leq\#\Gamma(K)) \\
+ & 0 & (\leq\#(\Gamma_Q(K)\setminus\ell(K))$ for $Q\in\ell(K)$ parabolic$) \\
+ & \frac{n}{2}\cdot(q-1) & (\leq\#(\Gamma_Q(K)\setminus\ell(K)))$ for $Q\in\ell(K)$ otherwise$) 
\end{array}
\]
Likewise we have
\[
\# T^\prime \leq \left(\frac{n}{2}-1\right)\cdot(q+1).
\]
Therefore
\[
  |T| \geq \left\{
  \begin{array}{l l l}
    \frac{q^2+2q+5}{2} & \quad \text{if $n=q-1$}\\
    \frac{q^2+3q+4}{2} & \quad \text{if $n=q$}\\
    \frac{q^2+4q+3}{2} & \quad \text{if $n=q+1$}
  \end{array} \right.
\]
and
\[
   |T^\prime| \leq \left\{
  \begin{array}{l l l}
    \frac{q^2-2q-5}{2} & \quad \text{if $n=q-1$}\\
    \frac{q^2-q-2}{2} & \quad \text{if $n=q$}\\
    \frac{q^2-1}{2} & \quad \text{if $n=q+1$}
  \end{array} \right..
\]
Notice that this gives us $|T|>|T^\prime|+q+1$ for all values of $n$, and that
\[
 |T|>\frac{1}{2}|S(K)|+\frac{q+1}{2} \iff |T|> |T^\prime|+q+1.
\] 
Thus $|T|>\frac{1}{2}|S(K)|+\frac{q+1}{2}$, from whence we can invoke Proposition~\ref{pigeon} to complete the proof.
\end{proof}


\begin{thebibliography}{}

\bibitem{Baker} H.\ F.\ Baker,
{\em Notes on the theory of the cubic surface},
Proc.\ London Math.\ Soc.\ {\bf 1} no.\ 9 (1910) 145--199.

\bibitem{Cooley} J.\ Cooley, 
{\em Generators for cubic surfaces with two skew lines over finite fields},
Arch. Math. (Basel) {\bf 100} (2013), no.\ 5, 401--411.

\bibitem{Geiser} K.\ F.\ Geiser,
{\em Ueber die Doppeltangenten einer eben Curve vierten Grades},
Math.\ Annalen {\bf 1} (1869) no.\ 1, 129--138.

\bibitem{HirshGorTor} J.\ W.\ P.\ Hirshfeld,\ G.\ Korchm\'aros,\ F.\ Torres,
{\em The number of points on an algebraic curve over a finite field}, 
London Math.\ Soc.\ Lecture Note Ser., {\bf 346}, Cambridge Univ. Press, Cambridge, (2007), 175--200




\bibitem{Ma1} Yu.\ I.\ Manin,
{\em Cubic Forms: Algebra, Geometry, Arithmetic},
North-Holland, 1974 and 1986.


\bibitem{Reid} M.\ Reid,
{\em Chapters on algebraic surfaces},
pages 1--154 of
{\em Complex algebraic varieties}, J. Koll\'ar (Ed.),
IAS/Park City lecture notes series (1993 volume), AMS, Providence R.I., 1997. 


\bibitem{Segre} B.\ Segre,
{\em A note on arithmetical properties of cubic surfaces},
J.\ London Math.\ Soc.\ {\bf 18} (1943), 24--31.


\bibitem{Siksek} S.\ Siksek,
{\em On the number of Mordell-Weil generators for cubic surfaces},
J.\ Number Theory {\bf 132} (2012), 260--2629.


\end{thebibliography}
\end{document}